\documentclass[12pt]{article}
\usepackage{amsmath,amsfonts,amssymb,mathptmx}
\usepackage{smallselmer,amscd,theorem}
\usepackage[all]{xy}
\usepackage{float}
\usepackage{graphicx}

\newtheorem{theorem}{Theorem}[section] 
\newtheorem{corollary}[theorem]{Corollary}
{\theorembodyfont{\rmfamily}   }
\newtheorem{lemma}[theorem]{Lemma}   \newtheorem{proposition}[theorem]{Proposition}

{\theorembodyfont{\rmfamily} \newtheorem{remark}[theorem]{\it Remark}  }
{\theorembodyfont{\rmfamily} \newtheorem{properTy}[theorem]{Property}  }
{\theorembodyfont{\rmfamily}   }

\renewcommand{\baselinestretch}{1.1}

\begin{document}

\pagestyle{smallselmer}
\thispagestyle{empty}
\begin{center}
\Large
Quadratic Twists of Elliptic Curves\\
with Small Selmer Rank\large\\
\rm
\rule{0pt}{1.5\baselineskip}\it
Sungkon Chang
 \end{center}

\begin{quote}\textbf{Abstract}\quad
Let $E/\ratn$ be an elliptic curve with no nontrivial rational $2$-torsion points where $\ratn$ is the rational numbers.
In this paper, we prove that there is a quadratic twist $E_D$ for which the rank of the $2$-Selmer group is less than or equal to $1$.
By the author's earlier result \cite{chang-thesis:2004}, this implies an unconditional distribution result on the size of the $2$-part of the Tate-Shafarevich group of the quadratic twists of $E$.  
By  \cite{monsky:1996}, if we assume the finiteness of the Tate-Shafarevich group of an elliptic curve, our result yields a fairly general distribution result for quadratic twists with Mordell-Weil rank $1$.
\end{quote}

\section{Introduction} \label{sec:intro}

Let $E/\ratn$ be an elliptic curve given by $y^2 = x^3 + ax + b$,  $E_D$, a quadratic twist $D\,y^2=x^3+ax+b$, and $\Sel2(E_D)$, the $2$-Selmer group of $E_D$.
In this paper, we prove
\begin{theorem}\label{maintheorem}\footnote{
			The result is improved recently by Mazur and Rubin \cite{mazur:2009} using Kramer's work \cite{kramer:1981}.}
	Let $E/\ratn$ be an elliptic curve with no nontrivial rational $2$-torsion points.  Then, 
		$$\Sharp\set{ \abs{D}< X : D\text{ square-free, } \dim\Sel2(E_D) \le 1 }
				\gg X/(\log X)^\al$$
				for some $0<\al<1$.
\end{theorem}
This result will follow from \cite[Theorem 1.2]{chang-thesis:2004} once we prove the existence of $D$ such that 
$\SEL(E_D) \le 1$.
Let  $\sha(E_D)[2]$ and $\rank\, E_D(\ratn)$ denote the $2$-part of the Tate-Shafarevich group and the Mordell-Weil rank, respectively.  Since  $\Sel2(E_D)\cong E_D(\ratn)/2E_D(\ratn)\oplus\sha(E_D)[2]$, our theorem implies the obvious distribution results for $\Sharp\sha(E_D)[2]$  and  $\rank\ E_D(\ratn)$.
The distribution result for  $\rank\, E_D(\ratn)=0$ is also obtained in \cite{ono:1998} and \cite{ono:2001} by establishing the non-vanishing of $L$-functions, 
but the one for $\Sharp\sha(E_D)[2]$ seems new when $E$ has no nontrivial rational $2$-torsion points.

\begin{corollary}\label{maincorollary}
Assume the finiteness of the Tate-Shafarevich group of all elliptic curves over $\ratn$.
Let $E/\ratn$ be an elliptic curve with no nontrivial rational $2$-torsion points such that $E$ has multiplicative reduction at $v \nmid 6\infty$.
Then, 
\begin{align*}
\Sharp\set{ \abs{D}< X : D\text{ square-free, } \rank\, E_D(\ratn) = 1\ \text{and}\ 
					\dim\sha(E_D)[2]&=0}\\
				&\gg X/(\log X)^\al\end{align*}
				for some $0<\al<1$ which depends on $E$.
\end{corollary}

If we assume the finiteness of the Tate-Shafarevich group,
the parity conjecture for elliptic curves over $\ratn$ can be  
proved using  the $2$-parity conjecture which was proved in \cite{monsky:1996}, and our corollary  follows from this result and the assumption.  The proof is given at the end of Section \ref{sec:proof}.

Our original goal was to (unconditionally) prove the existence of $D$ with $\dim\ \Sel2(E_D)=0$, but we were not able to push our method to obtain such a result.  
In fact, the finiteness of the Tate-Shafarevich group and the nondegeneracy of the Cassels-Tate pairing  imply that 
our method of twisting by large primes will fail to produce a $D$ such that $\SEL(E)=a$ and $\SEL(E_D)=b$ where $a+b=1$.  Our current goal is to remove the finiteness hypothesis in the corollary.

In \cite{heath-brown:1994}, for $E : y^2 = x^3 - x$ and a positive integer $n$, Heath-Brown computes  the proportion of $D$ up to $X$ as $X\to\infty$, for which $\SEL(E_D)=n$,  and in \cite{swinnerton:2008}, Sir Peter Swinnerton-Dyer obtains a similar result for far more general elliptic curves with full $2$-torsion points as the number of prime factors of  $D$ approaches infinity.
In  \cite{yu:2005}, using the $2$-Selmer groups of quadratic twists of  $E$ where $E$ is any elliptic curve with full rational $2$-torsion points, the author proves  conditional and unconditional results on the existence of a positive proportion of $D$ up to $X$ with $\rank E_D(\ratn)=a$ where $a\le 1$, and in \cite{xiong:2008} using a result of \cite{yu:2005}, they compute the average size of $\Sharp\,\sha(\widehat{E_D})[\hat\phi]$ where $\widehat{E_D}$ is the dual $2$-isogeny of $E$ considered in \cite{yu:2005}.   Via the modularity of elliptic curves over $\ratn$ and Kolyvagin's result, Ono and James in \cite{james-ono:1999} obtain results about the distribution of $D$ with the trivial $p$-Selmer group of $E_D$ where $p \ge 3$ and $E$ is a  fairly general elliptic curve.  Our method is rather direct and similar to \cite{swinnerton:2008}.  Using Schaefer's framework \cite{schaefer:1998}, we naturally identify each $\Sel2(E_D)$ with a subgroup of $\thegrouptwo{L}$ where $L$ is an \etale\ algebra over $\ratn$ that does not depend on $D$, and we do a rather explicit computation  in this fixed space $\thegrouptwo{L}$.

\par
In Section \ref{sec:local-cob}, we briefly summarize Schaefer's description of a Selmer group, and sketch the proof of our result.
Section \ref{sec:lemmas} consists of rather technical lemmas which shall be used to choose a prime number with properties required in the proof of our theorem.  The phenomena of the quadratic reciprocity law and the Cebotarev density theorem are the main tools for the proof of the lemmas in that section. In Section \ref{sec:proof}, we prove our theorem and corollary.

\paragraph{Acknowledgement}\quad 
I am indebted to Steve Donnelly for many valuable suggestions, and I would like to thank Sir Peter Swinnerton-Dyer for sharing his work.
I would like to also thank Mark Budden for his comments.
This work is partially supported by the Research and Scholoarship Fund of the Armstrong Atlantic State University.

\section{Computing the $2$-Selmer group}\label{sec:local-cob}

Recall that $E/\ratn$ is given by $y^2=x^3+ax+b$, and does not have nontrivial rational $2$-torsion points.   
Let $z_1$, $z_2$, and $z_3$ be the $x$-coordinates of the $2$-torsion points  in $\Qbar$, i.e., the roots of $x^3+ax+b$. Let  $L=L_E$ be the field extension $\ratn(z_1)$.  For each place $p$, let us denote by $n_p$ the number of places of $\OL$ lying over $p$.
Let $S_E$ be the set of places of $\ratn$ consisting of $\infty$ and $2$, and places of bad reduction of $E/\ratn$. 
By \cite[Proposition 3.4]{schaefer:1998}, we have the following isomorphism:
\begin{equation}
\label{eq:HonE}
\HonE(\ratn, E[2])_{S_E} \cong \ker\big( \normLQ : L({S_E},2) \To \ratn({S_E},2) \big).
\end{equation}
We refer to \cite{schaefer:1998} for the definition of $L(S_E,2)$.
For each finite or infinite place $p$, let $L_p:=L\otimes \Qp$. Then, we also have an isomorphism for the completion $\Qp$:
\begin{equation}\label{eq:local-H1}
\HonE(\Qp, E[2]) \cong \ker\big( \Norm_{L_p/\Qp} : \thegrouptwo{L_p} \To \thegrouptwo{\Qp} \big).
\end{equation}
These isomorphisms are defined with a choice of representatives of $\GQ$-orbits or $\Gal(\overline\Qp/\Qp)$-orbits in $E[2]$.  With certain choices of representatives, restriction maps $\res_p : \HonE(\ratn, E[2]) \to \HonE(\Qp, E[2])$ extend to the natural maps $\thegrouptwo{L} \to \thegrouptwo{L_p}$ which we also denote by $\res_p$ (see \cite[Proposition 2.4]{chang-thesis:2004}).
Since $L_p\cong \prod_{\primeP \mid p } L_\primeP$, later in our calculation, the map: $\thegrouptwo{L} \to \thegrouptwo{L_p}$ shall be interpreted as
\begin{equation}\label{eq:actual-map}
				[\al] \mapsto ( [\al] : \primeP \mid p).
				\end{equation}
The $2$-Selmer group can be described as follows:
\begin{equation}\label{eq:Sel-def}
\SeL(E)\cong \set{ \al \in L({S_E},2) : \normLQ(\al)=1,\ \res_p(\al)\in\Img\delta_p\text{ for all } p \in {S_E}}
\end{equation}
where $\delta_p$ is the coboundary map $\WMtwo[\Qp] \to \thegrouptwo{L_p}$ induced by the map $\HonE(\Qp,E[2]) \to \thegrouptwo{L_p}$ in (\ref{eq:local-H1}).

Let us identify the cohomology group $\HonE(\ratn,E[2])_{S_E}$ with the kernel in (\ref{eq:HonE}).  Let  $q_0=\infty$, and $q_1=2$, and write $S_E:=\set{q_0,q_1,\dots,q_n}$. Denote by $W_{-1}=W_{-1}^E$ the subgroup $\HonE(\ratn,E[2])_{S_E}$ of $L(S_E,2)$. 
For each $k=0,\dots,n$, let $W_{k}$ be the subgroup of $W_{k-1}$ consisting of $\al$ such that $\res_{q_k}(\al) \in \Img \delta_{q_k}$. Then, the $2$-Selmer group is isomorphic to 
$W_n$. 
Throughout the paper, by \lq\lq applying the local condition at $q_k$,\rq\rq\ we shall mean the process of obtaining $W_{k-1}$ from $W_k$.
Note that 
\begin{equation}\label{eq:dim-drop}
 \dim W_n = \dim W_{-1} - \sum_{k=0}^n \big( \dim W_{k-1} - \dim W_{k} \big).
 \end{equation}
Note also that $\res_{q_k}(W_{k-1})$ is a subgroup of $\HonE(\Qp[q_k],E[2])$, and we can write $\res_{q_k}(W_{k-1})=\res_{q_k}(W_{k})\oplus T$ for some subgroup $T$ of $\res_{q_k}(W_{k-1})$.
Since $T$ intersects $\Img\delta_{q_k}$ trivially, it follows that $\dim T + \dim\Img\delta_{q_k} \le \dim\HonE(\Qp[q_k],E[2])$ and, hence,
\begin{equation}\label{eq:local-upperbound}
\dim T =  \dim W_{k-1} - \dim W_{k} \le \dim\HonE(\Qp[q_k],E[2]) - \dim\Img\delta_{q_k}.
\end{equation}
By \cite[Corollary 3.6]{schaefer:1998}, we have the following possiblities for the upper bounds in (\ref{eq:local-upperbound}):
for all odd prime numbers $p$, we have $\dim \Img\delta_p = n_p-1$, and $\dim \HonE(\Qp,E[2])=2(n_p-1)$.
When $p=2$, we have $\dim \Img\delta_2 = n_2$, 
$\dim \thegrouptwo{L_2}=2n_2+3$, and $\dim \thegrouptwo{\ratn_2}=3$. 
By (\ref{eq:local-H1}), examining each case of $n_2$, we find that the description (\ref{eq:local-H1}) implies $\dim\HonE(\Qp[2],E[2]) = 2 n_2$.
Throughout the paper, we let
 $\ep:=0$ if $L$ has a complex embedding, and $\ep:=1$ if not.
When $p=\infty$, we have $\dim \HonE(\real,E[2]) = 2\ep$ and $\dim \delta_\infty = \ep$.
Hence, 
\begin{equation}\label{eq:downstep}
\dim\HonE(\Qp,E[2]) - \dim\Img\delta_p  = \begin{cases} 
					n_p-1, & p\text{ odd};\\
					n_2, & p=2;\\
					\ep, & p=\infty.
				\end{cases}
\end{equation}		

Let $\Cl_{S_E}(L)[2]$ denote the $2$-part of the $S_E$-class group of $L$.  Then, by the generalized Dirichlet unit theorem, 
\begin{align}
\dim W_{-1}
      &=\dim \HonE(\ratn,E[2])_{S_E} \notag\\
      &= (1+\ep)+\big(\sum_{k=1}^n n_{q_k} -1\big) + \dim \Cl_{S_E}(L)[2].\label{eq:dimension-H1}
\end{align}
If the equality in (\ref{eq:local-upperbound}) holds for all $k=0,\dots,n$, then  $\dim W_n = \dim \Cl_{S_E}(L)[2]$ and hence,  
$\SEL(E) = \dim \Cl_{S_E}(L)[2]$.  Our goal in this paper is to find a twist $E_D$ such that $\dim \Cl_{S_{E_D}}(L)[2]=0$ and $\dim W_{k-1}^{E_D} - \dim W_{k}^{E_D}$ (for $E_D$)  is maximized for each $k=0,\dots,n$. 
\par
By \cite[Proposition 3.1]{chang-thesis:2004}, $\dim\HonE(\ratn,E[2])_{S_E}=\dim \HonE(\ratn,E[2])_{S_E'}$ where
$S_E':=\set{ p \in S_E : n_p>1}\cup\set{2,\infty}$. Moreover, $\HonE(\Qp,E[2])=0$ if $p$ is an odd prime number with $n_p=1$.  So, let us redefine $S_E$ to be the set of places consisting of $\infty$, $2$, and places $p$ of bad reduction of $E$ with $n_p>1$.
\par
Now we describe the image of the local coboundary map in $L_p$.
First, we fix an embedding: $\Qbar \to \overline{\ratn}_p$.
Let $p$ be a prime number, and suppose that   $[\Qp(z_1):\Qp]\ge [\Qp(z_i):\Qp]$ for $i=1,\dots,3$. Then, 
  $$L_p \cong \begin{cases}
  				    \Qp\times\Qp\times\Qp, & n_p=3;\\
  				    \Qp(z_1)\times\Qp, & n_p=2;\\
  				    \Qp(z_1), & n_p=1.
  				    \end{cases}$$
Hence, if $p$ is an odd prime number, then  	
\begin{equation}\label{eq:H1-identification}
\HonE(\Qp,E[2]) \cong \begin{cases}
  				    \thegrouptwo{\Qp}\times\thegrouptwo{\Qp}, & n_p=3;\\
  				    \thegrouptwo{\Qp(z_1)}, & n_p=2;\\
  				    0, & n_p=1.
  				    \end{cases}
\end{equation}  				    
In practice, these products are interpreted as in (\ref{eq:actual-map}).
If $p=2$ with $n_p>1$, then the description of $\HonE(\Qp,E[2])$ is the same as above.  To give the description for the case $p=2$ with $n_p=1$, let $F:= \Qp(z_1)$.  Let $U_F$ be the group of unit integers in $F$, and $U_{\Qp}$, the group of unit integers in $\Qp$. Then, $\HonE(\Qp,E[2])$ is isomorphic to $\ker$ of the norm from $\thegrouptwo{U_F}$ to $\thegrouptwo{U_{\Qp}}$.  Since this norm map is surjective, and $\dim \thegrouptwo{U_F} = 4$ and $\dim\thegrouptwo{U_{\Qp}}=2$, we have 
		$\dim \HonE(\Qp,E[2]) = 2$.

If $\HonE(\Qp,E[2])$ is identified as in (\ref{eq:H1-identification}), then		
by \cite[Theorem 2.3]{schaefer:1998}, for each prime number $p$,
the local coboundary map $\WMtwo[\Qp] \to \HonE(\Qp,E[2])$ is given as follows:
$$(x,y) \mapsto ( x-z_i : i=1,\dots,n_p-1)$$
provided that $x\ne z_i$ for $i=1,\dots,n_p-1$.
\begin{lemma}\label{lem:themap}
If $n_p=3$, then under the local coboundary map at $p$, 
	\begin{align*} (z_1,0) &\mapsto ((z_1-z_2)(z_1-z_3), z_1-z_2);\\
			 (z_2,0) &\mapsto (z_2-z_1,(z_2-z_1)(z_2-z_3)).
	\end{align*}
If $n_p=2$ and $z_3\in\Qp$, then $(z_3,0) \mapsto (z_3-z_1)$.
\end{lemma}
\proof
The proof is left to the reader.\hfill\qed

\par
Let us conclude this section by sketching the proof of our main result.
Recall that $E[2](\ratn)$ is trivial.  Then, the cohomology groups $\HonE(\ratn,E_D[2])$ for all $D$ are naturally identified in $\thegrouptwo{L}$ where the \etale\ algebra $L/\ratn$ is a field extension of degree $3$, and similarly the local cohomology groups  $\HonE(\Qp[q],E_D[2])$ for all $D$ are identified in $\thegrouptwo{L_q}$ where $L_q=L\otimes \Qp[q]$; see  (\ref{diag:twist-diagram}) below. Moreover, we have the natural maps $\thegrouptwo{L} \to \thegrouptwo{L_q}$ which  restrict to the maps $\res_q : \HonE(\ratn,E_D[2]) \to \HonE(\Qp[q],E_D[2])$ for all $D$:
\begin{equation}\label{diag:twist-diagram}
\xymatrix@C=50pt{
		\HonE(\ratn,E_D[2]) \ar[r]\ar[d]_{\res_q} & \thegrouptwo{L} \ar[d]^{\res_q}\\
		\HonE(\Qp[q],E_D[2]) \ar[r] & \thegrouptwo{L_q}.
	}
\end{equation}
\par
For each $D$, let $S_D$ denote the set of places consisting of $\infty$, $2$, and places of bad reduction of $E_D$, and let $S=S_1$.
We find $D$ satisfying that $S \subset S_D$ which implies 
\begin{equation}\label{eq:Hone-inclusion}
\HonE(\ratn,E[2])_S \subset \HonE(\ratn,E_D[2])_{S_D}
\end{equation}
as subgroups of $\thegrouptwo{L}$, and that for each $q\in S$, the local coboundary images for  $E$ and $E_D$ identified in the space $\thegrouptwo{L_q}$ are \textit{equal to each other} (see \cite[Proposition 2.5]{chang-thesis:2004}).
By (\ref{eq:Hone-inclusion}) and the description (\ref{eq:Sel-def}),  what survives in $\HonE(\ratn,E_D[2])_{S_D}$ after applying local conditions over $S$ is a subgroup $W^D$ of $\thegrouptwo{L}$ \textit{containing}  $\Sel2(E)$ because of the second condition we impose above on $D$ (see (\ref{diag:Sel-ED}) in Section \ref{sec:proof}).
At each stage of applying local conditions as illustrated in (\ref{diag:Sel-ED}), the size of the subgroups surviving is decreased by at most the  numbers given in (\ref{eq:downstep}) in Section \ref{sec:local-cob}.
We find such $D$ as a product of two primes $p$ and $r$, where $p$ splits completely and $r$ remains prime in $L$, that 
    		\begin{equation}\label{eq:Hone-E-ED}
    		\dim \HonE(\ratn,E[2])_S +2 =\dim \HonE(\ratn,E_D[2])_{S_D},
    		\end{equation}
(see (\ref{eq:dimension-H1}) in Section \ref{sec:local-cob}).  The role of the prime $r$ shall be explained in the following paragraph.  Putting more conditions on $D$, we shall have, at some intermediate step of applying local conditions over $S$, the size of a subgroup surviving in $\HonE(\ratn,E_D[2])_{S_D}$ decreased by at least one more dimension than in $\HonE(\ratn,E[2])_{S}$; this is served by an element $\al$ contained in $\HonE(\ratn,E_D[2])_{S_D}\minuS\HonE(\ratn,E[2])_{S}$.  Then, it follows from (\ref{eq:Hone-E-ED}) that 
		\begin{equation}\label{eq:theinequality}
		\dim W^D  \le \SEL(E) + 1.
		\end{equation}
To compute $\SEL(E_D)$, we only need to apply the local condition at $p$ to $W^D$ since $\HonE(\Qp[r],E[2])=0$.  With more conditions imposed on $D$, we can map, via $\res_p$ the subgroup $\Sel2(E)$ of $W^D$ outside the local coboundary image, and it follows that  $\dim W^D -2  \le (\SEL(E) + 1)-2$, provided that $\SEL(E)\ge 2$ (see Figure 1 in Section \ref{sec:proof}).  That is, $\SEL(E_D) < \SEL(E)$.  By induction, we prove that there is some $D$ such that $\SEL(E_D)\le 1$. 

\par
It gets very technical to show that there is a $D$ satisfying all the properties we want, but the main tool is the generalized Dirichlet's theorem.  Recall that $\HonE(\ratn,E_D[2]) \to \HonE(\Qp[q],E_D[2])$ extends to the map: $\thegrouptwo{L} \to \thegrouptwo{L_q}$, the essence of which is the Legendre symbol over places of $L$ lying over $q$. 
Let us remark here that the \lq\lq quadratic residue properties\rq\rq\ of the element $\al$, which serves (\ref{eq:theinequality}), over the places $q\in S$ more or less determines the image of $\Sel2(E)$ in $\HonE(\Qp,E[2])$, via a quadratic reciprocity law; see Lemma \ref{lem:reciprocity}---this lemma is proved easily by Hecke's version of quadratic reciprocity for a number field \cite[Theorem 167]{hecke:1981}.
In the previous paragraph, we claimed that $\Sel2(E)$ lands outside the local coboundary image of $E_D$ at $p$, and it turns out that via  quadratic reciprocity this property  pleasantly follows  from the condition imposed on $D$ and hence on $\al$, which serves (\ref{eq:theinequality}).  Recall that $D=pr$.  In this very technical context, the inert prime $r$ serves to keep all local coboundary images of $E$ and $E_D$ over $q\in S$ being equal to each other, and this prime does not contribute to the size of $\HonE(\ratn,E_D[2])_{S_D}$; see (\ref{eq:dimension-H1}).  When $E[2](\ratn)\ne 0$, we do not have this prime, and it posed the main difficulty to extending our proof to the general case. 

\section{Lemmas}\label{sec:lemmas}

In this section, we introduce several lemmas which shall be used in the proof of Theorem \ref{maintheorem}, and 
 a slight generalization of the Legendre symbol.
Let $L$ be a field extension of $\ratn$ with degree $3$, and let $M$ be the Galois closure of $L$ over $\ratn$.  Let $K/\ratn$ denote the quadratic extension in $M$ if $L/\ratn$ is not Galois. 
For an odd prime ideal $\primep$ of $\OM$ and an element $\al$ of $\OM$ relatively prime to $\primep$,
we denote the Legendre symbol by $\legendretwo{\al}{\primep}$, and for an odd element $\beta$ relatively prime to $\al$, denote the Jacobi symbol by $\legendretwo{\al}{\beta}$.
If $v$ is a real embedding of $M$, then $\legendretwo{ \al }{ v }$ is defined to be $1$ if $v(\al)>0$, and $-1$ if not. 

We  define an extension of the Legendre symbol over even primes as follows:
Let $\primep$ be a  prime ideal dividing $2\OM$.  For $\al \in \OM$ coprime to $\primep$, we define the symbol 
$ \legendretwo{ \al }{ \primep }$ to be the class of $\al$ in $\grouptwo{U_\primep}$ where $U_\primep$ is the group of unit integers in the completion of $M$ at $\primep$.
Suppose that $\primep$ and $\primep'$ are prime ideals dividing $2\OM$ and $\mu \primep'=\primep$ for some $\mu \in \Gal(M/\ratn)$. Then there is a canonical isomorphism $\grouptwo{U_{\primep'}} \to \grouptwo{U_\primep}$ such that the following diagram  is commutative:
\begin{equation*}
\xymatrix{
	(\OM/(\primep')^n)^* \ar[r]^\mu\ar[d] & (\OM/(\primep)^n)^* \ar[d]\\
	\grouptwo{U_{\primep'}} \ar[r] & 	\grouptwo{U_\primep}
}
\end{equation*}
where the vertical maps are surjective for sufficiently large $n$, and given by the Legendre symbols.
Let us denote by $\legendretwo{ \al }{\primep'}_\mu$ the image of $\legendretwo{ \al }{\primep'}$ in $\grouptwo{U_\primep}$.
This definition is consistent with the usual Legendre symbol over an odd prime ideal $\primep$ since $\thegrouptwo{U_\primep}\cong \zz/2\zz$, and 
$\legendretwo{\al}{\primep}\legendretwo{\al}{\primep'}_\mu$ can be simply denoted by 
$\legendretwo{\al}{\primep}\legendretwo{\al}{\primep'}$.

Suppose that $L/\ratn$ is not Galois.
If $\primeP$ is an even prime of $L$, and $\primeQ\subset\OM$ is an unramified prime over $\primeP$ with residue degree $f(\primeQ/\primeP)=1$, then
we have a canonical isomorphism $\grouptwo{U_\primeP} \to \grouptwo{U_\primeQ}$.  Under this isomorphism, we may write $\legendretwo{\al}{\primeQ}=\legendretwo{\al}{\primeP}$ for $\al \in L^*$, and this abuse of notation is used in Lemma \ref{lem:action}.

\begin{lemma}\label{key-lemma} 
Let $T$ be the set of $2$ and prime divisors of the square-free part of $\Delta_{K/\ratn}=(z_1-z_2)^2(z_2-z_3)^2(z_3-z_1)^2$.
Let $R=\set{r_i : i=1,\dots,m}$ and $\set{t_i : i=1,\dots,m}$ be  a sequence of prime numbers and a sequence of $\pm 1$, respectively, such that $R$ does not intersect $T$.  Then, there is an odd prime number $p$ such that 
$n_p=1$ and $\legendretwo{p}{r_i}=t_i$ for $i=1,\dots,m$, and such that $\legendretwo{p}{q}=1$ for all $q\in T$. 
\end{lemma}

\begin{proof}
Suppose that $L/\ratn$ is not Galois.
Let $K$ be the quadratic extension in the Galois closure of $L$, and let $K':=K(\sqrt{q},\sqrt{-1} : q \in T)$.
Let $F/K'$ be the field extension $K'(\sqrt{r_i} : i=1,\dots,m)$.
Then, since $R$ intersects $T$ trivially, $\Gal(F/K')=\oplus_{i=1}^m \zz/2\zz$.
Write $T=\set{2,q_1,\dots,q_s}$.  Note that $K$ is $\ratn(\sqrt{\pm d})$ where $d$ is $2\prod_1^s q_i$ or $\prod_1^s q_i$, and that  $LK'/K'$ and $F/K'$ are linearly disjoint and Galois.  Hence, $\Gal(LF/K')\cong\Gal(LK'/K')\oplus\Gal(F/K')$. Let $F_i:=K'(\sqrt{r_i})$.
Then, there is an automorphism $\tau$ in $\Gal(LF/K')$ such that  $\res_{LK'}(\tau)$ is a generator of $\Gal(LK'/K')$ and such that $\res_{F_i}(\tau)\cong\inner{t_i}$ for each $i=1,\dots,m$.
 By the the Cebotarev density theorem, there is a prime number $p\ne 2$  whose Frobenius automorphism in $\Gal(LF/\ratn)$ is $\tau$. 
 Since $K'/\ratn$ is Galois, $p$ splits completely in $K'$.
 By our choice of $\tau$, it follows that  $n_p=1$, and $\legendretwo{r_i}{p}=t_i$ for all $i$. Since $\legendretwo{-1}{p}=1$ (i.e., $p\equiv 1 \mod 4$), we have $t_i=\legendretwo{p}{r_i}$.
 Moreover, $1=\legendretwo{q}{p} =\legendretwo{p}{q}$ for all odd primes $q \in T$.
 If $q=2$, then by the supplementary reciprocity law, 
  $p\equiv 1 \mod 8$ if and only if $\legendretwo{2}{p}=1$ (provided that $p\equiv 1 \mod 4$). 
  Since $2$ is contained in $T$, and $p$ splits completely in $K'$, we have $\legendretwo{2}{p}=1$ and, hence, $\legendretwo{p}{2}=1$.
  The proof of the case that $L/\ratn$ is Galois is similar, and rather simpler.
  \hfill\qed
\end{proof}

\begin{remark}\label{rem:choose}
Let $T$ be a finite set of infinite or finite places of $M$. It is well-known that for any modulus $\idealM$ and a prime power decomposition $\idealM=\prod_{\primeQ\in T} \idealM_\primeQ$, 
\begin{equation}\label{eq:CRT}
\rayclass{M}{\idealM} \cong \bigoplus_{\primeQ\in T} \rayclass{M}{\idealM_\primeQ}
\end{equation}
where $M_\idealn=\set{\al\in M^* \text{ coprime to }\idealn}$, and $M_{\idealn,1}$ denotes the \textit{ray mod $\idealn$} (i.e., $\set{ \al \in M_\idealn : \al \equiv^* 1 \mod \idealn }$).
 If $H$ is the Hilbert class field of $M$, by the class field theory,
$\Gal(R/H) \cong (\rayclass{M}{\idealM}) / \OM^*$, as shown below, where $\OM^*$ is identified with its image in $\rayclass{M}{\idealM}$:
\begin{equation}\label{diag:CFT}
\xymatrix{
	1\ar[r] & \Gal(R/H)\ar[r] & \Gal(R/M)\ar[r] & \Gal(H/M)\ar[r] & 1 \\
	1\ar[r] &   J\ar[r]\ar@{<->}[u] & I_\idealM/P_{\idealM,1}\ar[r]\ar@{<->}[u] 
											& I/P\ar[r] \ar@{<->}[u] & 1 
}
\end{equation}
where $J=(\rayclass{M}{\idealM}) / \OM^*$, $I_\idealM$ is the group of fractional ideals of $\OM$ coprime to $\idealM$, $P_{\idealM,1}$ is the  group of principal fractional ideals $\beta\OM$ such that $\beta \equiv^* 1 \mod \idealM$, and 
$I/P$ is the ideal class group of $\OM$.

\begin{lemma}\label{lem:choose-rho}
Let $M$, $H$, and $\idealM$ be as above.  Then, there is a prime element $\rho$ of $\OM$ lying over a prime number $p$ which splits completely in $M$ such that $\rho$ belongs to any class in 
$\rayclass{M}{\idealM}$.
\end{lemma}
\begin{proof}
Let $R$ be the ray class field of $M$ mod $\idealM$, and $R'$ be the Galois closure of $R$ over $\ratn$, so $\Gal(R'/M) \subset \Gal(R'/\ratn)$.
Let $\al$ be a number in $M_\idealM$ belonging to an arbitrary class in $\rayclass{M}{\idealM}$,
and let $\sig$ be an automorphism in $\Gal(R/H)$ corresponding to the class $[\al] \in \rayclass{M}{\idealM}$ via the map in (\ref{diag:CFT}).
Let $\sig'$ be the automorphism in $\Gal(R'/\ratn)$ which restricts to $\sig$.
By the Cebotarev density theorem, there are a prime ideal $\primeP'$ of $\OK[R']$ and a prime number $p$ such that $\Frob(\primeP' / p) = \sig' \in \Gal(R'/H)$.
Let $\primeP:=\primeP' \cap \OK[R]$ be the prime ideal of $\OK[R]$, and $\primep:=\primeP \cap \OM$,
the prime ideal of $\OM$. Then,
$\id = \Frob(\primeP'/p) \mid_M = \Frob(\primep/p)$, and hence, the residue degree $f(\primep/p)=1$.
Thus, $p$ splits completely in $M$.

Note that $\sig' = \Frob(\primeP'/p) = \Frob(\primeP'/p)^{f(\primep/p)} = \Frob(\primeP'/\primep)$, and hence,
$\sig=\Frob(\primeP/\primep)$.
By (\ref{diag:CFT}), the prime ideal $\primep$ is a principal ideal $\rho'\OM$ such that 
		$[\rho']=[\al]$ in $(\rayclass{M}{\idealM}) / \OM^*$.
So, $\rho' \equiv^* \al\lambda \mod \idealM$ for some $\lambda \in \OM^*$.
Let $\rho:=\rho' \lambda\Inv$. Then, $\rho \equiv^* \al \mod \idealM.$\hfill\qed
\end{proof}

Let $T_0:=\set{\primep_1,\dots,\primep_t}$ be the set of all places of $M$ over $2$.
Let $T:=\set{\primeQ_1,\dots,\primeQ_s}$ be a finite set of (finite or infinite) places of $M$ outside $2$, and let $\set{\ep_\primeQ : \primeQ \in T}$ be a sequence of $\pm 1$ indexed over $T$. Let  $\idealM$ be a modulus supported by $T\cup T_0$ with large exponents over places in $T_0$. 
By Lemma \ref{lem:choose-rho}, we can choose a prime element $\rho$ of $\OM$ lying over a prime number $p$ which splits completely in $M$ such that
	\begin{alignat}{2}\label{eq:choose-rho}
	\legendretwo{\rho}{\primeQ} & =\ep_\primeQ &&\quad\text{for all } \primeQ \in T;\\
	\legendretwo{\rho}{\primep}  & =1 &&\quad \text{for all } \primep \in T_0.\notag
	\end{alignat}
	
\end{remark}	
	
\begin{lemma}\label{lem:action}
Suppose that $L/\ratn$ is not Galois.
Let $\sig$ be an automorphism in $\Gal(M/\ratn)$ with order $3$.
Let $\tau$ be the generator of $\Gal(M/L)$.
Let $\rhoone$ be a  prime element of $\OM$ over an odd prime $p$ which splits completely in $M$, and $\rho_2:=\sig \rhoone$ and $\rho_3:=\sig^2\rhoone$.  Let $\al_i:=\Norm_{M/L}(\rho_i)=\rho_i\cdot\tau\rho_i$.
 Let $\primeQ_1$ be a  prime ideal of $\OM$ over an even or odd prime $q$ not equal to $p$.
 Suppose that $q$ splits completely in $M$, and let $\primeQ_2:=\sig \primeQ_1$ and $\primeQ_3:=\sig^2\primeQ_1$.  Let $\primeP_i\OM:=\primeQ_i\cdot \tau\primeQ_i$ for some prime ideal $\primeP_i$ of $\OL$.
 \par
  Then,  $\normLQ(\al_1\al_2)$ and $\normLQ(\al_2\al_3)$ are contained in $(\ratn^*)^2$, and 
 \begin{align}
 \albe11 &= \rhoQtwo11{}{\tau}, &  \albe21 &= \rhoQtwo32\sig{\tau\sig},\notag\\
 \albe12 &= \rhoQtwo22{}{\tau}, &  \albe22 &= \rhoQtwo13{\sig}{\tau\sig},\notag\\
 \albe13 &= \rhoQtwo33{}{\tau}, &  \albe23 &= \rhoQtwo21{\sig}{\tau\sig},\notag
 \end{align}
 \begin{align*}
 \albe31 &= \albetwo23{\sig} \albetwo11{} \albetwo22{\sig^2},\\
 \albe32 &= \albetwo21{\sig} \albetwo12{} \albetwo23{\sig^2},\\
  \albe33 &= \albetwo11{\sig} \albetwo23{} \albetwo12{\sig^2}.
 \end{align*}
 Suppose that $q$ has $n_q=2$, i.e., $q\OL = \primeP_1^2 \primeP_2$ or $q\OL = \primeP_1\primeP_2$ with $f(\primeP_1/q)=2$.
 Hence, $\primeP_1\OM = \primeQ_1\cdot \tau \primeQ_1
 $ and $\primeP_2\OM=\primeQ_2^2$ or $\primeP_2\OM=\primeQ_2$.
 Then, 
 \begin{equation*}
\begin{aligned}
  \albe11 &= \rhoQtwo11{}{\tau} \\
  \albe21 &= \legendretwo{\rhoone}{\primeQ_3}_{\sig}\legendretwo{\rhoone}{\primeQ_1},
 \end{aligned}\ \text{ or }\ %
 \begin{aligned}
   \albe11 &= \rhoQtwo11{}{\tau} \\
   \albe21 &= \legendretwo{\rhoone}{\primeQ_3}_{\sig}\legendretwo{\rhoone}{\primeQ_2},
 \end{aligned}
 \end{equation*}
 if $\sig \primeQ_1 = \tau\primeQ_1$ or $\sig \primeQ_1 = \primeQ_2$, respectively.
 \end{lemma}

 \begin{proof}
Note that
		\begin{align*}
		\normLQ(\al_1\al_2)  &=\prod_{i=1}^3 \sig^i (\al_1\al_2)
						=\prod_{i=1}^3\sig^i(\rho_1\cdot\tau\rho_1)\ %
							\prod_{i=1}^3\sig^i(\sig\rho_1\cdot\tau\sig\rho_1)\\
						&=\left( \prod_{\mu\in\Gal(M/\ratn)} \mu(\rho_1)\ \right)^2=p^2.
		\end{align*} 
 Simiarly, $\normLQ(\al_2\al_3)$ is a square in $\ratn^*$.
 
 Since $\albe11=\legendretwo{ \al_1 }{ \primeQ_1}$, it follows that 
 	$$\albe11 = \legendretwo{ \rhoone }{ \primeQ_1}\legendretwo{ \tau\rhoone }{ \primeQ_1}
 	=\rhoQtwo11{}{\tau}.$$
As the proof of the other results is very similar, we leave the details to the reader.\hfill\qed
\end{proof}

The relationship between the Legendre symbols for the case of $L/\ratn$ being Galois is simpler, and we state it below without proof.
\begin{lemma}\label{lem:Galois}
Suppose that $L/\ratn$ is Galois, i.e., $M=L$.
Let $\sig$ be a nontrivial automorphism of (order $3$) in $\Gal(L/\ratn)$.
Let $\al_1$ be a prime element of $\OL$ over an odd prime $p$ which splits completely in $L$, and $\al_2:=\sig \al_1$ and $\al_3:=\sig^2\al_1$.  
 Let $\primeP_1$ be a  prime ideal of $\OL$ over an even or odd prime $q$ not equal to $p$.
 Suppose that $q$ splits completely in $L$, and let $\primeP_2:=\sig \primeP_1$ and $\primeP_3:=\sig^2\primeP_1$.  Then,
 $\normLQ(\al_1\al_2)$ and $\normLQ(\al_2\al_3)$ are contained in $(\ratn^*)^2$, and 
\begin{alignat*}{2}
 \albetwo11{} & =\albetwo22{\sig^2} & &  =\albetwo33{\sig};\\
 \albetwo12{} & =\albetwo23{\sig^2} & &=\albetwo31{\sig};\\
 \albetwo13{} & =\albetwo21{\sig^2} & &=\albetwo32{\sig}.
 \end{alignat*}
 \end{lemma}

\begin{remark}\label{rem:free-var}
If $L/\ratn$ is not Galois, the symbols $\legendretwo{\rho_1}{\primeQ_i}$  and $\legendretwo{\rho_1}{\tau\primeQ_i}$ in Lemma \ref{lem:action} can be treated as free variables as 
explained in (\ref{eq:choose-rho}).  
Then, the matrix over $\zz/2\zz$ associated with the system of the first six equations in Lemma \ref{lem:action} has rank $5$, and $\prod_{i=1}^2\prod_{j=1}^3\legendretwo{\al_i}{\primeP_j} =1$ is the (only) constraint. 
Thus, the five variables $\legendretwo{\al_i}{\primeP_j}$ for $i<2$ or $j<3$
can be treated as free variables.

If $L/\ratn$ is Galois, by Lemma \ref{lem:Galois}, $\albe1j$ for $j=1,2,3$ can be treated as free variables.
\par
A similar result is available for real places of $L$ if there are three real embeddings of $L$.   Suppose that $L$ has three real embeddings, and that $\Gal(M/\ratn)=\inner{\sig,\tau}$ where $\sig$ has order $3$, and $\tau$ has order $2$.  Then, $M$ must have six real embeddings, and we let $v$ be a real embedding of $M$.
All six embeddings are in the form of $v\mu$ for some $\mu \in \Gal(M/\ratn)$.  It is clear that $\legendretwo{ x }{ v\mu }=\legendretwo{ \mu(x) }{ v }$ for all $\mu \in \Gal(M/\ratn)$, 
and if $u$ is a real embedding of $M$ and $u'$ is the restriction of $u$ to $L$, then $\legendretwo{ x }{ u }=\legendretwo{ x }{ u' }$ for all $x \in L$.  It follows that
$$\legendretwo{ \al_1\al_2 }{ v }=\rhoQv{}\rhoQv{\tau}\rhoQv{\sig}\rhoQv{\tau\sig}.$$
Likewise, $\legendretwo{ \al_i\al_j }{ u }$ where $u$ is a real embedding of $M$ are written in terms of the Legendre symbols of $\rho_1$ over real embeddings, and we can obtain, by imposing values on $\legendretwo{\rho_1}{u}$'s,  the desired values of $\legendretwo{ \al_i\al_j }{ u }$ which are required for  the proof of our theorem.
The case: $M=L$ (with $3$ real embeddings) is simpler.
When there is only one real embedding of $L$, the restriction map $\HonE(\ratn,E[2]) \to \HonE(\real,E[2])$ is trivial as $\HonE(\real,E[2])$ is trivial by (\ref{eq:local-H1}).  
\end{remark}
	
\begin{lemma}\label{lem:reciprocity}\ (Reciprocity Law) 
\  Let $M/\ratn$ be a Galois extension.
Let $x$ and $y$ be elements of $\OM$ such that $x\OM = \primeQ_1\cdots\primeQ_s \idealA^2$ and  
	$y\OM = \primeQ_1'\cdots\primeQ_t' \idealB^2$ where $\primeQ_i$ and $\primeQ_j'$ are primes, and such that $xy\OM$ is not a square.
Let $\idealM_x$ be the product of $\primeQ_i$ which are odd, for $i=1,\dots,s$, and 	
$\idealM_y$, the product of $\primeQ_i'$,  which are odd, for $i=1,\dots,t$.
If $\rho$ is a totally positive prime element $\rho \in \OM$, coprime to $xy$, lying above an odd prime number $p$ splitting completely in $M$ such that 
$\rho \equiv^* 1 \mod 8\infty\OM$, then
	\begin{align*}
	\legendretwo{ x }{ \rho } &=  \prod_{\primeQ \mid \idealM_x} \legendretwo{\rho}{\primeQ};\\
	\legendretwo{ y }{ \rho } &= \prod_{\primeQ \mid \idealM_y} \legendretwo{\rho}{\primeQ}.
	\end{align*}
\end{lemma}
\begin{proof}
Let $\idealM$ be the lcm of $\idealM_x$ and $\idealM_y$, and  $\al$ be any integer in $M$ coprime to $\idealM$ such that 
$$\al \equiv^* 1 \mod 8\infty\OM.$$
Then, $\sgn(\sig\, \al) = 1$ for all real embeddings $\sig : M \to \real$, and 
by \cite[Theorem 167]{hecke:1981}, we have 
		$$\legendretwo{x}{\al} = \legendretwo{\al}{\idealM_x},\quad
			\legendretwo{y}{\al} = \legendretwo{\al}{\idealM_y}.$$
\hfill\qed
\end{proof}

Note that $L\otimes \Qp[2]$ is a product of local fields.
Let $\HonE(\Qp[2],E[2])_0$ denote the subgroup generated by tuples in $L\otimes \Qp[2]$ whose entries are unit integers of the local fields.
\begin{proposition}
There are totally positive elements $\beta_1,\dots,\beta_t$ in $L^*$ lying over primes $p$ with $n_p=3$,  representing elements in $\HonE(\ratn,E[2])$
such that $\inner{\beta_1,\dots,\beta_t}$ surjects down to $\HonE(\Qp[2],E[2])_0$ under the restriction map.
\end{proposition}

\begin{proof}
Suppose that $n_2>1$, and that $L/\ratn$ is not Galois.
Using Lemma \ref{lem:action}, we have
\begin{enumerate}
\item\label{case1} if $2\OM=\primeP_1\primeP_2\primeP_3$, then
		\begin{align}
		\legendretwo{\al_1\al_2}{\primeP_1}
			&= \rhoQtwo11{}{\tau}\rhoQtwo32\sig{\tau\sig},\notag\\
		 \legendretwo{\al_1\al_2}{\primeP_2} 
		 	&= \rhoQtwo22{}{\tau}\rhoQtwo13{\sig}{\tau\sig},\notag
 \end{align}
\item\label{case2} if $2\OM=\primeP_1^2\primeP_2$ or $\primeP_1\primeP_2$, then 
 \begin{equation*}
 \legendretwo{\al_1\al_2}{\primeP_1} 
   =\begin{cases}
   \legendretwo{\rhoone}{\tau\primeQ_1}_{\tau} 
       \legendretwo{\rhoone}{\primeQ_3}_{\sig},\\
	 \rhoQtwo11{}{\tau} \legendretwo{\rhoone}{\primeQ_3}_{\sig}\legendretwo{\rhoone}{\primeQ_2}.       
	 \end{cases}
 \end{equation*}
 \end{enumerate}	
 
Note that $(\OM/\primeP^n)^* \to \thegrouptwo{U_\primeP}$ is surjective for  $n$ sufficiently large.
Consider the modulus $\idealM=2^n\infty\OM$, and the isomorphism
  		$$ \rayclass{M}{\idealM} 
  		 \cong \bigoplus_{\primeP \mid \idealM} \rayclass{M}{\primeP^{ne}}
  		 \oplus\ \bigoplus_{v \mid \infty} M_v/ M_{v,1}.$$
Suppose that $2$ splits completely in $L$. 
As mentioned earlier in Remark \ref{rem:free-var}, all $\legendretwo{\rhoone}{\primeQ_i}$ and $\legendretwo{\rhoone}{\tau\primeQ_i}$ can be treated as free variables.  Moreover, $\rhoone$ can be chosen to be a totally positive prime element lying over a prime number with $n_p=3$. So, recall the case \ref{case1} above, and let
		$$\legendretwo{\rhoone}{\primeQ_1}=\legendretwo{\rhoone}{\tau\primeQ_2}=1.$$
Then, 		
\begin{align*}
		\legendretwo{\al_1\al_2}{\primeP_1}
			&= \legendretwo{\rhoone}{\tau\primeQ_1}_\tau
			    \legendretwo{\rhoone}{\primeQ_3}_\sig\\
		 \legendretwo{\al_1\al_2}{\primeP_2} 
		 	&= \legendretwo{\rhoone}{\primeQ_2}
			    \legendretwo{\rhoone}{\tau\primeQ_3}_{\tau\sig},
 \end{align*}
and it is clear that we can impose any pair of  values on $\legendretwo{\al_1\al_2}{\primeP_1}$ and 
$\legendretwo{\al_1\al_2}{\primeP_2}$.  
Recall that $\al_1\al_2$ represents an element in $\HonE(\ratn,E[2])$, and also recall the description (\ref{eq:actual-map}) and (\ref{eq:H1-identification}).  Then, it can hit an arbitrary element in $\HonE(\Qp[2],E[2])_0$ by choosing values of the Legendre symbols on $\rho_1$.
Thus, we can find the $\beta_i$'s claimed in the theorem.  The proof of Case \ref{case2} is similar.

Suppose that $L/\ratn$ is Galois, and that $2\OM$ splits completely. 
In this case, we have only three free variables as shown in Lemma \ref{lem:Galois}, and we have
\begin{align*}
	\legendretwo{\al_1\al_2}{\primeP_1}
			&=\albe11\albe13_{\sig}\\
	\legendretwo{\al_1\al_2}{\primeP_2}		
		&=\albe12\albe11_{\sig}.
\end{align*}		
So, let $\albe11=1$. Then, it is clear that any pair of values can be imposed on $\legendretwo{\al_1\al_2}{\primeP_1}$ and $\legendretwo{\al_1\al_2}{\primeP_2}$,
which proves the result for the case of Galois extensions.
\par
Suppose that $n_2=1$.
Note that there are only two extensions of $\Qp[2]$ of degree $3$, namely, 
$\fieldL_1:=\Qtwo[x]/(x^3-2)$ and $\fieldL_2:=\Qtwo[x]/(x^3-x+1)$. The local field $\fieldL_1$ corresponds to the case where $2$ is totally ramified in $L$, and $\fieldL_2$, the case where $2$ is inert in $L$.
Consider the case $2\OL=\primeP$.  Regardless of whether or not $L/\ratn$ is Galois, the completion $L_\primeP$ is Galois over $\Qtwo$, and it is $\fieldL_2$. Let $\primeP\OM=\primeQ\cdot\tau\primeQ$ be the factorization of $\primeP$ in $M$ if $L/\ratn$ is not Galois.  Then, an automorphism $\sig \in \Gal(M/\ratn)$ of order $3$ 
acts trivially on $\set{\primeQ,\tau\primeQ}$, and this automorphism can be considered as one in $\Gal(\fieldL_2/\Qtwo)$.  Note that $\sig\in\Gal(\fieldL_2/\Qtwo)$ acts on $\UBtwo\cong\UBtwo[\primeQ]$ while it was not the case for $n_p>1$.  Since $n_2=1$, the kernel of the map
		$ \Norm_{\fieldL_2/\Qtwo} : \thegrouptwo{\fieldL_2} \to \thegrouptwo{\Qtwo}$
is the kernel of 	
		$$ \Norm_{\fieldL_2/\Qtwo} : \UBtwo \To \UBtwo[2]$$
which we denote by $V$, and $\dim V = 2$.  We claim that $\sig$ acts nontrivially on $V$ and hence, nontrivially on the corresponding subgroup of $\UBtwo[\primeQ]$, which we denote by $V'$.
Suppose that $\sig$ acts trivially on $V$ so that for each $\al \in U_\primeP$ representing an element of $V$, there is some $x \in U_\primeP$ such that $\sig \al = \al x^2$.
Then since $[\al] \in V$, there is $y \in U_2$ such that $y^2=\al\cdot\sig\al \cdot \sig^2 \al
			=\al \cdot \al x^2 \cdot \al x^2 \sig(x)^2$, and hence,
			$\al = \big( y (\al x^2 \sig(x))\Inv \big)^2 \in U_\primeP^2$.  This proves the claim.
Note that , \small
\begin{equation*}
		\legendretwo{\al_1\al_2}{\primeP}=
		\begin{cases}
			 \legendretwo{\rhoone}{\primeQ}
			     \legendretwo{\rhoone}{\tau\primeQ}_\tau
			      \legendretwo{\rhoone}{\primeQ}_\sig
			       \legendretwo{\rhoone}{\tau\primeQ}_{\tau\sig} & 
			       \text{if $L/\ratn$ is not Galois}\\
		  \legendretwo{\al_1}{\primeP}
		 			     \legendretwo{\al_1}{\primeP}_\sig
			       & \text{if $L/\ratn$ is  Galois.}
			    \end{cases}
 \end{equation*}\normalsize
 Recall that we can treat $\legendretwo{\rhoone}{\primeQ}$, $\legendretwo{\rhoone}{\tau\primeQ}$, and  $\legendretwo{\al_1}{\primeP}$ as free variables on $\UBtwo[\primeQ]$ or on $\UBtwo[\primeP]$.
 If we write these vector spaces over $\finitefield[2]$ additively, the tranformation $(1+\sig)$ is invertible on $V'$ since $\sig(1+\sig)=\sig +\sig^2=-1=1$. This means that each of $\legendretwo{\rhoone}{\primeQ} \legendretwo{\rhoone}{\primeQ}_\sig$ and $\legendretwo{\rhoone}{\tau\primeQ}\legendretwo{\rhoone}{\tau\primeQ}_{\sig}$ can assume any value in $V'$, and so does  $\legendretwo{\al_1}{\primeP}\legendretwo{\al_1}{\primeP}_\sig$  in $V$.  Thus, using the Cebotarev density theorem as in (\ref{eq:choose-rho}), we obtain the result.
 \par
 Consider the case $2\OL=\primeP^3$. As in the previous case, for $\HonE(\Qtwo,E[2])$, we only need to consider the kernel, denoted also by $V$, of 
 		$$ \Norm_{\fieldL_1/\Qtwo} : \UBtwo \To \UBtwo[2],$$
and $\dim V = 2$.  Note that $\primeP\OM=\primeQ$, and that the natural map $\UBtwo \to \UBtwo[\primeQ]$ is not injective.  It has kernel of dimension $1$, and it is the subgroup of $\UBtwo$ generated by the class represented by $-3$ as $U_\primeQ$ has a primitive third root of unity.  Since  $\UBtwo=V\oplus\inner{[-3],[3]}$, it implies that $V$ injects into $\UBtwo[\primeQ]$; we denote the image of $V$ by $V'$.
Let $\curlyM$ be the completion of $M$ at $\primeQ$.
Let $\Gal(\curlyM/\Qtwo)$ be generated by $\sig$ of order $3$ and $\tau$ which generates $\Gal(\curlyM/\fieldL_1)$.  Note that $\Gal(\curlyM/\Qtwo)$ acts on $\UBtwo[\primeQ]$.  
As in the previous case, involved in computing $\legendretwo{\al_1\al_2}{\primeP}$ is the transformation $1 + \tau + \sig + \tau\sig$ on $\UBtwo[\primeQ]$.
Note that $1 + \tau + \sig + \tau\sig$ restricted to $V'$ is $1+1+\sig +\sig^2=\sig+\sig^2$ since $\tau$ acts trivially on $V'$.  As in the previous case, it turns out that  $\sig$ acts nontrivially on each of the three nontrival elements of $V'$ and, hence,  $\sig+\sig^2=-1=1$ on $V'$.  Therefore, if we choose a totally positive prime element $\rho_1$ of $\OM$ such that $\legendretwo{\rho_1}{\primeQ}$ is any value $z$ in $V'\subset\UBtwo[\primeQ]$, then
$$\legendretwo{\al_1\al_2}{\primeQ}=
		\legendretwo{\rhoone}{\primeQ}
					     \legendretwo{\rhoone}{\primeQ}_\tau
					      \legendretwo{\rhoone}{\primeQ}_\sig
			       \legendretwo{\rhoone}{\primeQ}_{\tau\sig}
			 =z.$$
Since $V$ is naturally isomorphic to $V'$ and $[\al_1\al_2]\in V$, it means that $\legendretwo{\al_1\al_2}{\primeP}$ can be any value in $V\subset\UBtwo[\primeP]$.\hfill\qed
 \end{proof}

\begin{corollary}\label{cor:surjection}
There is a set $S$ consisting of $\infty$, $2$, and primes $p$ with $n_p=3$ such that 
$\HonE(\ratn,E[2])_S$ surjects down to $\HonE(\Qp[2],E[2])$ under the restriction map.
\end{corollary}

\begin{proof}
Let $\primeP$ be a prime ideal of $\OL$.  
If $L/\ratn$ is not Galois, then $M/L$ is ramified and, hence, the Hilbert class field $H$ of $L$ and $M/L$ are linearly disjoint.
It follows that $\Gal(HM/M) \cong \Gal(H/L)$ which is true for the case $M=L$, as well.
Then, by choosing a Frobenius automorphism in $\Gal(HM/M)$, we can find a prime ideal $\primep_1$ of $\OL$ representing the class of $\primeP$ in $\Cl(\OL)$ such that the prime number $p:=\primep_1\cap\zz$ splits completely in $L$. From $p\OL=\primep_1\primep_2\primep_3$, it follows that $\primeP\primep_2\primep_3$ is a principal ideal.

Suppose that $2\OL=\primeP_1\primeP_2\primeP_3$.  Then, there are prime numbers $p$ and $p'$ splitting completely in $L$, i.e.,
$p\OL=\primep_1\primep_2\primep_3$ and $p'\OL=\primep_1'\primep_2'\primep_3'$, such that 
$\primeP_1\primep_2\primep_3$ and $\primeP_2\primep_2'\primep_3'$ are generated by $\al$ and $\beta$, respectively.
Note that $\normLQ(\al\beta)=\pm(2pp')^2$. So, either of $x:=\pm \al\beta$ represents an element of $\HonE(\ratn,E[2])$ since $[L:\ratn]=3$.  Then, it is clear that we can find an element $y$ representing an element in $\HonE(\ratn,E[2])$ such that $y\OL=\primeP_2\primeP_3\idealA$ where $\idealA$ is square-free and supported only by prime ideals lying over primes splitting completely.   The two elements $\res_2(x)$ and $\res_2(y)$ together with $\HonE(\Qp[2],E[2])_0$ generate $\HonE(\Qp[2],E[2])$, and by the previous proposition, we prove the result.  The proof for the other cases of splitting of $2\OL$ is similar.\hfill\qed
\end{proof}

\section{Proof}\label{sec:proof}
We prove Theorem \ref{maintheorem} in this section.  Let $E$ and $L$ be as in Section \ref{sec:local-cob}.
From now on, when $q$ is an odd prime number, we say that $\Img \delta_q$ is \textit{unramified} if $\Img \delta_q$ is contained in $\HonE(\Qp[q],E[2])\unr = L_q(\emptyset,2)$, and \textit{totally ramified} if $\Img \delta_q$ contains no nontrivial elements of $L_q(\emptyset,2)$. For example, if $n_p=3$, then a subgroup $\inner{(pa,b),(c,pd)}$ in $\HonE(\Qp,E[2])$ is totally ramified where $a,b,c,d$ are unit integers, while $\inner{(pa,b),(pc,d)}$ is not totally ramified. 
Let $S$ be the set of $2$, $\infty$, and places $q$ of bad reduction of $E/\ratn$ such that $q$ is ramified in $L$ with $n_q=2$. This set $S$ depends only on $L$.   In this section, our elliptic curve is replaced by its twists several times, but the set $S$ is not replaced, and \textit{fixed to be the one just described here, once and for all}.
Let ${T_E}$ be the set of places consisting of $\infty$, $2$, and places $p$ of bad reduction of $E$ with $n_p>1$, except for the odd primes $q \in S$ such that $\Img \delta_q$ is unramified.

First, we shall argue below that there is a quadratic twist  $E'$ satisfying the following properties:
\begin{properTy}\label{property-E} \hspace{1em}%Assume that the elliptic curve $E$ satisfies
\begin{enumerate}
\item $E'$ has good reduction at odd places $p$ with $n_p=2$ which are unramified in $L$, and $\Img \delta_p$ is totally ramified at odd places $p$ of bad reduction with $n_p=3$;
\item $\Cl(L)[2]$ is generated by prime ideals over primes  $p$ with $n_p=3$ at which $E'$ has bad reduction;
\item $\SeL(E')$ intersects $\HonE(\ratn,E'[2])_S$ trivially;
\item As we apply the local conditions to $\HonE(\ratn,E'[2])_{T_{E'}}$ over $\infty$, $2$,  and all odd primes $q\in S$, 
	the dimension drops by $\ep$ at $\infty$, and by $1$ at odd primes $q \in S$ if $\Img \delta_q$ is ramified.
\item $\HonE(\ratn,E'[2])_{T_{E'}}$ surjects onto $\HonE(\Qp[2],E'[2])$.	
\end{enumerate}
\end{properTy}

Let us first show that if $p$ is an odd prime number unramified in $L$ such that $n_p=2$ then either $E$ or the quadratic twist $E_p$ has good reduction at $p$.
Note that a twist $E_D$ can be given by $y^2=x^3+aD^2x + bD^3$.
Twisting by $p$, if necessary, we may assume that $x^3+ax+b=((x-z_3))((x-z_3) - (z_1-z_3))((x-z_3) - (z_2-z_3))$ where $z_3\in\Qp$ such that $\ord_\primeP(z_1-z_3)=\ord_\primeP(z_2-z_3)$ is even where $p\OL = \primeP\primeP'$ and $f(\primeP/p)=2$.  Then, changing the variables over $\Qp$, we find a $\zz_p$-model 
$y^2=x ( x - \al ) ( x - \beta)$ where $(x-\al)(x-\beta)$ is an irreducible polynomial over $\Qp$, and $\al$ and $\beta$ are unit integers.  Since $n_p=2$, the quadratic polynomial remains irreducible over $\finitefield$ and, hence, it has good reduction at $p$.
 If $n_p=3$, then either $E$ or the twist $E_D$ with $D=p$ has a totally ramified local coboundary image at $p$.  This can be easily shown by considering the cases of the triples of parities of $\ord_p(z_1-z_2)$, $\ord_p(z_2-z_3)$, and $\ord_p(z_3-z_1)$ and using Lemma \ref{lem:themap}.  So, let us assume without loss of generality Property \ref{property-E}, \#1.
\par
Note that we may as well assume that $\Cl(L)[2]$ is generated by prime ideals $\primeP$'s dividing primes  $p$ with $n_p=3$ at which $E$ has bad reduction.  We always have such primes $p$ as argued in the proof of Corollary \ref{cor:surjection}, and after twisting $E$ with the product of these primes $p$, it has Property \ref{property-E}, \#2.
\par
Let $x$ be an element of $\OL$ representing a nontrivial element in $\HonE(\ratn,E[2])$.  Then, $L(\sqrt{x})$ is linearly disjoint with $M$ over $L$ since $M=L(\sqrt{\Delta})$ with $\Delta \in \ratn$.   Using  the Cebotarev density theorem and  the nontrivial automorphism in $\Gal(LM(\sqrt{x})/M)$, we can find a prime $p$ splitting completely in $L$, i.e., $p\OL=\primeP_1\primeP_2\primeP_3$, such that $\legendretwo{x}{\primeP_1}=-1$.  
Recall that given an elliptic curve $E/\ratn$, we denote by $S_E$ the set of $2$, $\infty$, and places $p$ of bad reduction of $E/\ratn$ with $n_p>1$.   Since $\HonE(\ratn,E[2])_S$ is finite, we can find $k$ distinct prime numbers $p \not \in S_E$ with $n_p=3$, where $k=\Sharp\ \HonE(\ratn,E[2])_S -1$, such that 
each nontrivial element $z\in\HonE(\ratn,E[2])_S$ represented by an element $x$ of $L^*$ is a non-square unit integer in some $L_p$.  Let $D$ be the product of these $k$ primes. Then, by Lemma \ref{lem:themap}, the local coboundary images of $E_D$ at these primes $p$ is totally ramified, and each nontrivial element $z$ of $\HonE(\ratn,E[2])_S$ is mapped to a nontrivial element of $\HonE(\Qp,E_D[2])\unr$ at some of these primes $p$ as illustrated below:
\begin{center}
\includegraphics{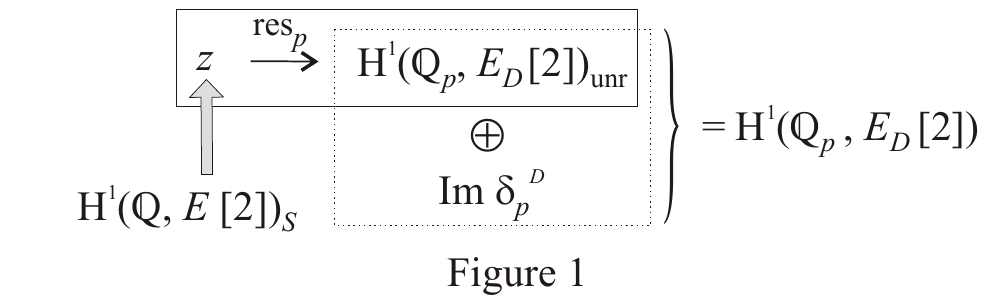}
\end{center}
where $\Img \delta_p^D$ denotes the image of the local coboundary map at $p$ for $E_D$.
Hence, $\SeL(E_D)$ intersects $\HonE(\ratn,E[2])_S$ trivially; see Property \ref{property-E}, \#3.  In particular, $\SeL(E_D)$ does not contain an element represented by  $x \in L^*$ such that $x\OL$ is a square.  Thus, we can assume that if $x\in\OL$ represents an element in $\SeL(E)$, then the square-free part of $x\OL$ is necessarily supported by a prime ideal outside $S$.
\par
Let us recall Lemma \ref{lem:action}, Lemma \ref{lem:Galois}, (\ref{eq:choose-rho}), and Remark \ref{rem:free-var}. Given an odd ramified prime $q$ with $n_q=2$ and $q\OL = \primeP_1^2\primeP_2$, we can find a prime number $p$ such that $\legendretwo{\al_1\al_2}{\primeP_1}=-1$, and $\legendretwo{\al_1\al_2}{\primeP}=1$ for all other places $\primeP$ lying over $2$, $\infty$ or a prime number $r$ of bad reduction of $E$ with $n_r>1$ where $\al_1$ and $\al_2$ are as defined in Lemma \ref{lem:action} and \ref{lem:Galois}.    For each of these primes $q$, let $p_q$ denote the prime number $p$, and we renew our elliptic curve $E$ with the twist $E_D$ where $D$ is the product of the $p_q$'s.  Then, as we apply the local conditions at odd primes $q \in S$ to $\HonE(\ratn,E[2])_{S_E}$, some elements $\al_1\al_2$ of $\HonE(\ratn,E_D[2])$ are mapped to the nontrivial element of $\HonE(\Qp[q],E[2])\unr$ under $\res_q$ as in Figure 1 and, hence, at odd primes $q \in S$,  the dimension drops  by $1$ if $\Img\delta_q$ is (totally) ramified; see (\ref{eq:downstep}) and  Property \ref{property-E}, \#4.  This twisting step possibly changes the local coboundary images everywhere, but satisfied are the properties which are discussed in the previous paragraphs, and stated in Property \ref{property-E} as the first three items. 
Note that Property \ref{property-E}, \#3 remains true since after twisting, $\Img\delta_q^D$ may be different from $\Img\delta_p$, but remains totally ramified.
 If $\Img\delta_q$ is unramified when $q \in S$ is odd, then $\Img \delta_q=\inner{ (a) }$ where $a$ is a  unit integer of $\Qp[q]$, and no elements in $\HonE(\ratn,E[2])$ which are ramified at $q$ will  be mapped to $\Img \delta_q$ under $\res_q$.  This means that, to compute $\SeL(E)$, we may apply the local conditions to $\HonE(\ratn,E[2])_{T_E}$.  Since $\Cl(L)[2]$ is generated by prime ideals over  primes $p$ with $n_p=3$, at which $E$ has bad reduction, by (\ref{eq:dimension-H1}), we have 
	$$\dim\HonE(\ratn,E[2])_{T_E} = \dim \HonE(\ratn,E)_{S_E} 
	       - \Sharp (S_E \minuS T_E).$$
More importantly, since the $\al_i\al_j$ in this paragraph are still contained in $\HonE(\ratn,E)_{T_E}$, the fourth item in Property \ref{property-E} is satisfied for odd primes $q \in S \cap T_E$.
By Remark \ref{rem:free-var}, when $L$ has three real embeddings, we have a version of Lemma \ref{lem:action} and \ref{lem:Galois} for $q=\infty$, and can find a prime number $p_\infty$ such that $\inner{\al_1\al_2} \to \HonE(\real,E[2])$ is surjective.
By Corollary \ref{cor:surjection}, we may assume Property \ref{property-E}, \#5.

\par
We shall show that if ($E$ satisfies Property \ref{property-E}, and) $\SEL(E) \ge 2$, then $\SEL(E_D) < \SEL(E)$ for some $D$.
Suppose that  $\SEL(E) \ge 2$.
Then, by Property \ref{property-E}, \#3, the subgroup $\SeL(E)$ contains two elements represented by $x$ and $y$ in $L^*$ such that the square-free part of each $x\OL$ and $y\OL$ is supported by ideals lying over splitting primes $q_1$ and $q_2$ outside $S$, which are not necessarily distinct.
To compute $\SeL(E)$,  we shall apply to $\HonE(\ratn,E[2])_{T_E}$ the local conditions at $q\in S$ first, and then at the remaining places, in an order such that the local conditions at $q_1$ and $q_2$ are applied at the very end as shown in (\ref{eq:Figure2}).  Recall from Sec \ref{sec:local-cob} the definition of $W_k$'s:
\begin{equation}\label{eq:Figure2}
\xymatrix@R=.9\baselineskip{
	\text{\makebox[2em][r]{$\HonE(\ratn,E[2])_{T_E}$}} &
		W_0 \ar[l]_{\incl} \ar[d]_{\res_\infty} 
			& \cdots\ar[l]_{\incl} 
				& W_{n-1} \ar[l]_{\incl} \ar[d]_{\res_{q_1}}  
					& W_n\text{\makebox[0pt][l]{\ $=\Sel2(E)$.}}
						\ar[l]_{\incl} \ar[d]_{\res_{q_2}}
						& \HS{1.5em}\\
		& \Img \delta_\infty 
			&
				& \Img \delta_{q_1}
					& \Img \delta_{q_2}
					   &\HS{1.5em}
				}
\end{equation}				
	Of course, the value of $\SEL(E)$ does not depend on the order we apply local conditions, but we use this order to make our later argument work.
Note that by Property \ref{property-E}, \#4 and \#5, there must be a place $q^*\in S_E$ with $n_{q^*}=3$ at which the dimension does not drop by $2$  as we apply the local conditions in an  order described above; otherwise, $\SEL(E)$ would be zero.
\begin{lemma}\label{lem:choose}
Suppose that $\SEL(E) \ge 2$. Then, there are $x$ and $y$ in $\Sel2(E)$ and $q_1$ and $q_2$ outside $S$ with $q_1\OL=\primeP_1\primeP_2\primeP_3$ and $q_2\OL=\tilde\primeP_1\tilde\primeP_2\tilde\primeP_3$
such that $q^*=q_2$, or $q^*$ is not equal to $q_1$ or $q_2$, and that
\begin{equation}\label{eq:factorization1}
\begin{aligned}
		x\OL & =\primeP_1\primeP_2\big(\primePtilde_1^{d_1}\primePtilde_2^{d_2}\primePtilde_3^{d_3}\big)
				\idealA_x \idealB_x^2;\\
		y\OL &=\primePtilde_1\primePtilde_2\big(\primeP_1^{e_1}\primeP_2^{e_2}\primeP_3^{e_3}\big)
				\idealA_y \idealB_y^2,
				\end{aligned}\end{equation}
where $\idealA_x$ and $\idealA_y$ are square-free ideals,$d_i$ and $e_i$ are $0$ or $1$, $\sum d_i \equiv \sum e_i \equiv 0 \mod 2$,  $d_1d_2 \ne 1$, and $e_1 e_2 \ne 1$, or
\begin{equation}\label{eq:factorization2}
\begin{aligned}
		x\OL & =\primeP_1\primeP_2
				\idealA_x \idealB_x^2;\\
		y\OL &=\primeP_2\primeP_3
				\idealA_y \idealB_y^2,
				\end{aligned}\end{equation}
where $\idealA_x$ and $\idealA_y$ are square-free ideals supported by $S$.
\end{lemma}
\begin{proof}
Let us prove the assertion that $q^*=q_2$, or $q^*$ is not equal to $q_1$ or $q_2$.
Suppose that $q^*=q_1$ and $q_1\ne q_2$.  This means that the dimension drops by $2$ at all places prior to $q_1$ and $q_2$; otherwise, we could choose $q^*$ not equal to $q_1$ or $q_2$.  Note that if $\dim W_{n-1} - \dim W_n < 2$, then
we can choose $q^*=q_2$, and we are done.  So, assume that $\dim W_{n-1} - \dim W_n = 2$, which means
$\res_{q_2}(W_{n-1}) = V \oplus T$ where $V$ is the maximal subgroup contained in $\Img\delta_{q_2}$ and $\dim T = 2$. 
Note that since $W_{n-1} \subset W_{n-2}$, we have 
	$\res_{q_2}(W_{n-1}) \subset \res_{q_2}(W_{n-2})$.
This means that if we apply the local condition at $q_2$ earlier than at $q_1$, then
$V \oplus T \subset \res_{q_2}(W_{n-2})$ and, hence, the dimension of $W_{n-2}$ drops  by $2$ after applying the local condition at $q_2$. Thus, $q^*=q_1$ must be the case.
So, we may assume that the local conditions at $q_1$ are applied first, and $q^*=q_2$.

Suppose that we can choose $q_1\ne q_2$, so we have the factorization in (\ref{eq:factorization1}).
If for all such $q_1$ and $q_2$, we have  $d_1 d_2 = 1$ and $e_1 e_2 = 1$, then we must have 
$xy \in \HonE(\ratn,E[2])_S$ which contradicts Property \ref{property-E}, \#3.
So, we have $d_1 d_2 \ne 1$ or $e_1 e_2 \ne 1$ for some choice of distinct $q_1$ and $q_2$.
Suppose that $d_1 d_2 = 1$ and $e_1 e_2 \ne 1$. Then,
\begin{align*}
		xy\OL & =\primeP_1^{e_1+1}\primeP_2^{e_2+1}\primeP_3^{e_3}
				\idealA_x\idealA_y \idealB_x^2\idealB_y^2;\\
		y\OL &=\primePtilde_1\primePtilde_2\big(\primeP_1^{e_1}\primeP_2^{e_2}\primeP_3^{e_3}\big)
				\idealA_y \idealB_y^2.
				\end{align*}
If $x'=xy$ and $y'=y$, then for all possibilities of $e_1$, $e_2$, and $e_3$, the factorization of $x'$ and $y'$ satisfies the conditions stated for (\ref{eq:factorization1}).

It is clear that if we can not choose distinct $q_1$ and $q_2$, then (\ref{eq:factorization2}) must be the case. 
\hfill\qed
\end{proof}

Our goal here is to find a prime number $p$ with $p\OL = \al_1\al_2\al_3 \OL$ such that Property \ref{property-al} below is satisfied:
\begin{properTy}\label{property-al} \HS{2em}
\begin{enumerate}
\item $\inner{ \al_1\al_2, \al_2\al_3} \to \HonE(\Qp[q^*],E[2])\unr$ is surjective;
\item $\inner{ x, y } \to \HonE(\Qp,E[2])\unr$ is surjective;
\item $\inner{ \al_1\al_2, \al_2\al_3} \to \HonE(\Qp[q],E[2])$ 
	is the trivial map for all other places $q$ in $S_E$.
\end{enumerate}
\end{properTy}
Using this prime $p$, we will  find a twist such that $\SEL(E_D) < \SEL(E)$.

By Lemma \ref{lem:choose}, we have the following possibilites:
(a) $q_1$, $q_2$, and $q^*$ are all distinct; (b) $q_1 \ne q_2 = q^*$; 
(c) $q_1=q_2=q^*$; (d) $q_1=q_2\ne q^*$.
Since the proofs of (a) and (b), and of (c) and (d) are very similar to each other,
we present the proofs of (a) and (c) in this paper.

Suppose that $L/\ratn$ is not Galois.
Let us consider the case that $q_1=q_2=q^*$.  By Lemma \ref{lem:choose}, we have factorization 
(\ref{eq:factorization2}). We claim that there are prime elements $\al_i$ of $\OL$ for $ i=1,2,3$ lying over a prime $p$  such that $\al_i\al_j$'s represent elements in $\HonE(\ratn,E[2])$, and 
\begin{enumerate}
\item \begin{equation}\label{eq:atq-prime}
	\legendretwo{ \al_1 \al_2 }{ \primeP }=1
	\end{equation}
              for all  places $\primeP$ dividing places  contained in $S_E$, not equal to $q^*$;
\item 
	\begin{equation}\label{eq:showering}
	\begin{aligned}
	\legendretwo{ \al_1 \al_2 }{ \primeP_1 } &= -1, \quad& 
			\legendretwo{ \al_2 \al_3 }{ \primeP_1 } &=1,\\
	\legendretwo{ \al_1 \al_2 }{ \primeP_2 } &= 1, \quad& 
			\legendretwo{ \al_2 \al_3 }{ \primeP_2 } &=-1.
	\end{aligned}
	\end{equation}
\end{enumerate}
 By Lemma 
\ref{lem:action},
\begin{align*}
	\legendretwo{ \al_2 \al_3 }{ \primeP_1 }
		&= \albe21 \albe23 \albe11 \albe22;\\
	\legendretwo{ \al_2 \al_3 }{ \primeP_2 }
		&= \albe22 \albe21 \albe12 \albe23.
\end{align*}
Recall  from Remark \ref{rem:free-var} the (only) constraint on the values of Legendre symbols above, and use the Cebotarev density theorem as in (\ref{eq:choose-rho}) to choose a totally positive prime element $\rho_1$ such that $\albe11=-1$, $\albe23=-1$, $\albe{i}{j}=1$  for all other pairs $(i,j)$, and $\legendretwo{ \rho_1 }{ \primeP }=1$
              for all  places $\primeP$ dividing places  contained in $S_E$ not equal to $q^*$, so that (\ref{eq:atq-prime}) and (\ref{eq:showering}) are satisfied. Then, by Lemma \ref{lem:action}, and the reciprocity law, Lemma \ref{lem:reciprocity}, we have
\begin{align}
	\legendretwo{x}{\al_1} &= \legendretwo{x}{\rhoone}  
					    =\prod_{\primeQ \mid \idealM_x} \legendretwo{\rho_1}{\primeQ}
						=\prod_{\primeQ \mid \primeP_1\primeP_2\OM} 
											\legendretwo{\rho_1}{\primeQ}\notag\\
					&=\rhoQ11 \rhoQ22 = \albe11 \albe12\notag\\
					&=-1;\label{eq:x-al1}\\
	\legendretwo{x}{\al_2} &= \legendretwo{x}{\rho_2}  
					  =\prod_{\primeQ \mid \idealM_x} \legendretwo{\rho_2}{\primeQ}
						=\prod_{\primeQ \mid \primeP_1\primeP_2\OM} 
						   						\legendretwo{\rho_2}{\primeQ}\notag\\
					&=\prod_{\primeQ \mid \primeP_1\primeP_2\OM} 
						   		\legendretwo{\rho_1}{\sig^2\primeQ}\notag\\	   						
					&=\rhoQ32 \rhoQ13 = \albe21 \albe 22\notag\\
					&=1.\label{eq:x-al2}
\end{align}
Similarly, we find 
\begin{align}\label{eq:eq-y}
	\legendretwo{y}{\al_1} &=\albe12\albe13=1;\\
	\legendretwo{y}{\al_2} &=\albe22\albe23=-1.\notag
\end{align}
\par
Let us consider the case that $q^*$, $q_1$, and $q_2$ are all distinct.
Let 		$q^*\OL=\primeP_1^*\primeP_2^*\primeP_3^*$ be the factorization into prime ideals.
By Lemma \ref{lem:choose},
we may assume 
\begin{align*}		
		x\OL & =\primeP_1\primeP_2\big(\primePtilde_1^{d_1}\primePtilde_2^{d_2}\primePtilde_3^{d_3}\big)
				\big((\primeP_1^*)^{s_1}(\primeP_2^*)^{s_2}(\primeP_3^*)^{s_3}\big)
				\idealA_x \idealB_x^2;\\
		y\OL &=\primePtilde_1\primePtilde_2\big(\primeP_1^{e_1}\primeP_2^{e_2}\primeP_3^{e_3}\big)
				\big((\primeP_1^*)^{t_1}(\primeP_2^*)^{t_2}(\primeP_3^*)^{t_3}\big)
				\idealA_y \idealB_y^2\\
			&\HS{2.5em}
			\textstyle\text{where }\sum d_i \equiv \sum e_i \equiv \sum s_i\equiv \sum t_i \equiv 0 \mod 2,\\
			&\HS{2.5em}
			 \text{and }d_1 d_2 \ne 1 \text{ and } e_1 e_2 \ne 1.
				\end{align*}
Let $T_{ij}:=\legendretwo{\al_i}{\primeP_j}$, $\Ttilde_{ij}:=\legendretwo{\al_i}{\primePtilde_j}$, and 
			$\Tstar_{ij}:=\legendretwo{\al_i}{\primePstar_j}$.
Then, as in (\ref{eq:x-al1}), (\ref{eq:x-al2}), and (\ref{eq:eq-y}), we have 
\begin{equation}\label{eq:system}
\begin{aligned}\
\legendretwo{x}{\al_1} &= T_{11} T_{12}
				\big(  \Ttilde_{11}^{d_1}\Ttilde_{12}^{d_2}\Ttilde_{13}^{d_3} \big)
					\big(  (\Tstar_{11})^{s_1}(\Tstar_{12})^{s_2}(\Tstar_{13})^{s_3} \big);\\
\legendretwo{x}{\al_2} &=T_{21} T_{22}
				\big(  \Ttilde_{21}^{d_1}\Ttilde_{22}^{d_2}\Ttilde_{23}^{d_3} \big)
					\big(  (\Tstar_{21})^{s_1}(\Tstar_{22})^{s_2}(\Tstar_{23})^{s_3} \big);\\
\legendretwo{y}{\al_1} &=\Ttilde_{11} \Ttilde_{12}
				\big(  T_{11}^{e_1}T_{12}^{e_2}T_{13}^{e_3} \big)
					\big(  (\Tstar_{11})^{t_1}(\Tstar_{12})^{t_2}(\Tstar_{13})^{t_3} \big);\\
\legendretwo{y}{\al_2} &=\Ttilde_{21} \Ttilde_{22}
				\big(  T_{21}^{e_1}T_{22}^{e_2}T_{23}^{e_3} \big)
					\big(  (\Tstar_{21})^{t_1}(\Tstar_{22})^{t_2}(\Tstar_{23})^{t_3} \big).
\end{aligned}
\end{equation}
Recall that (\ref{eq:showering}) describes the map $\inner{\al_1\al_2, \al_2\al_3} \to \HonE(\Qp[q^*],E[2])\unr$, and
it is clear that there are values of  $\Tstar_{ij}$ such that this map is surjective.
Now, with these values, let us treat $\Tstar_{ij}$ as constants in the above system where $T_{ij}$ and $\Ttilde_{ij}$ are variables.
 If all $d_i$ and $e_j$ are zeros, then the above system has rank $4$.
If all $e_i$ are zeros, but some $d_j$ is not, then it is clear that the system has rank $4$.
Say $e_1=0$, and $e_2=e_3=1$, and treat $\Ttilde_{ij}$ as constants. Then, the rank of the system is greater than or equal to  the rank of the vectors
		$$ T_{11}T_{12},\  T_{21}T_{22},\  T_{12}T_{13},\  T_{22}T_{23}$$
where $T_{22}T_{23}=T_{11}T_{12}T_{13}T_{21}$, and the five variables $T_{ij}\ne T_{23}$ are free.
The set of these four vectors has rank $4$.
Similarly, when $e_2=0$ and $e_1=e_3=1$, it turns out that the system has rank $4$.  
This proves that we could choose $p$ such that $\inner{x,y} \to \HonE(\Qp,E[2])\unr$ is surjective.
\par
The case that $L/\ratn$ is Galois can be handled in the similar way.
In fact, we considered the system (\ref{eq:system}) to compute the image of $\inner{x,y}$ under  the restriction map at $p$ in $\finitefield[2] \times \finitefield[2]\cong\HonE(\Qp,E_D[2])\unr$.
For the Galois case, we have only three free variables $T_{11}$, $T_{12}$, and $T_{13}$ (see Lemma \ref{lem:Galois}).
For the case $q_1=q_2=q^*$, as in (\ref{eq:system}), the restriction map at $p$ is given by
		\begin{align*}
		\big( \legendretwo{x}{\al_1}, \legendretwo{x}{\al_2}\big)
				&= \big( T_{11}T_{12}, T_{21}T_{22} \big)\\
	\big( \legendretwo{y}{\al_1}, \legendretwo{y}{\al_2}\big)
				&= \big( T_{12}T_{13}, T_{22}T_{23} \big)			
				\end{align*}
where $T_{21}=T_{13}$, $T_{22}=T_{11}$, and $T_{23}=T_{12}$.
By a similar analysis as in the non-Galois case, it turns out that for all cases, with these three free variables for each prime, the restriction map at $p$ is surjective. 
\par
Suppose we have Property \ref{property-al}.
Recall that we chose $\rho_1$ using (\ref{eq:CRT}) in Section \ref{sec:lemmas}, which
determined the $\al_i$'s above.
Note that using a higher power modulus $\idealM$ in (\ref{eq:CRT}), we can choose $\rho_1$
which yields the $\al_i$'s with the same property given above.
Let $K'=M(\sqrt{q} : q<\infty,\ q\in S)$. Then, using a modulus $\idealM$ of high exponents, we may assume that $K'$ is contained in the ray class field $R_\idealM$ of $M$ mod $\idealM$.
Recall that we choose $\rho_1$ in $\rayclass{M}{\idealM}$ such that $\rho_1 \equiv^* 1 \mod \idealM'$
where $\idealM'$ is the \lq\lq greatest divisor\rq\rq\ of $\idealM$ supported only by $S$ 
as in (\ref{eq:atq-prime}).
As $p$ denotes the rational prime over which $\rho_1$ lies, $\Frob(\rho_1)$ in $\Gal(R_\idealM/M)$ restricts trivially to $K'/M$, i.e., $p$ splits completely in $K'$, which means that $\legendretwo{q}{p}=1$ for all finite primes $q\in S$.  Since $2\in S$, we have $p\equiv 1 \mod 8$ and, hence, $\legendretwo{p}{q}=1$ for \textit{all} finite primes $q\in S$.
Then, by Lemma \ref{key-lemma}, there is an inert prime $D_0$ such that 
for  all finite primes $r \in S_E$, we have $\legendretwo{p}{r}=\legendretwo{D_0}{r}$, i.e., 
$\legendretwo{pD_0}{r}=1$.  Now let $D=pD_0$.  Then, at each $r$ contained in $S_E$, the local coboundary image of $E_D$ is equal to that of $E$ at $r$ (see \cite[Proposition 2.5]{chang-thesis:2004}).  
Let  $\TED$ be $T_E\cup\set{p}$.
Then, $\dim \HonE(\ratn, E[2])_{T_E} + 2 = \dim \HonE(\ratn, E_D[2])_{\TED}$, and the following diagram illustrates the computation of $\SeL(E_D)$:
\begin{equation}\label{diag:Sel-ED}
\xymatrix@C=50pt{
	\HonE(\ratn,E[2])_{T_E} \ar[r]^{\incl}\ar[d]^{\Psi_\infty^{E}} 
				& \HonE(\ratn,E_D[2])_{\TED}\ar[d]^{\Psi_\infty^{E_D}}\\
	 \vdots \ar[d]^{\Psi_{q}^{E}} 
	          & \vdots \ar[d]^{\Psi_{q}^{E_D}} \\
	 \SeL(E)   \ar[r]^{\incl}\ar[d]^{\Psi_p^{E_D}}  
	 		& W_n^{E_D}\ar[d]^{\Psi_p^{E_D}}  \\
	  V\ar[r]^{\incl} & \SeL(E_D)}
\end{equation}	  
where the maps $\Psi_v^{E_D}$ and $\Psi_v^E$ denote the process of applying the local condition at $v$, and $n$ is  a positive integer.
Note that if $v\ne p$ is contained in $\TED$, then $\Img\Psi_v^{E}\subset\Img\Psi_v^{E_D}$ since 
$\Img\delta^D_v = \Img\delta_v$, and that  (\ref{eq:showering}) makes the dimension drop by $2$ at $q^*$.
Since $\Img\Psi_v^{E}\subset\Img\Psi_v^{E_D}$, at each step, the dimension drops by more or the same dimension for $\HonE(\ratn,E_D[2])_{\TED}$ than for $\HonE(\ratn,E[2])_{T_E}$.
Especially, at $q^*$, it necessarily drops by more dimension for $\HonE(\ratn,E_D[2])_{\TED}$.
Recall (\ref{eq:dim-drop}) from Section \ref{sec:local-cob}.
Then, if $h_D$ and $h$  denote the sums of dimension-drops over $T_E$ for $\HonE(\ratn,E_D[2])_{\TED}$ and $\HonE(\ratn,E[2])_{T_E}$, respectively, then we have
\begin{align*}
\dim W_n^{E_D} &= \dim \HonE(\ratn,E_D[2])_{\TED} - h_D\\
	&=\dim\HonE(\ratn,E[2])_{T_E} + 2 - h_D \\
	&<\dim\HonE(\ratn,E[2])_{T_E} + 2 - h =\SEL(E)+2.
\end{align*}	
Thus, $\dim W_n^{E_D} \le \SEL(E) +1$.
Since $\Img\delta_p^D$ is totally ramified and $\inner{x,y}$ maps onto $\HonE(\Qp,E_D[2])\unr$,  the dimension drops by $2$ at $p$.
Therefore, $\SEL(E_D)=\dim W_n^{E_D} - 2 \le \SEL(E) - 1$ and, hence, 
$\SEL(E_D) < \SEL(E)$.  This concludes the proof of the existence of such $D$.

\par
When $q^*$ is not equal to $q_1$ or $q_2$, the order of applying the local conditions, which was specified  earlier (see (\ref{eq:Figure2})), matters for our argument.   By (\ref{eq:atq-prime}), the elements $\al_i\al_j$ satisfy all the local conditions of $E_D$ and $E$ at all places in $S_E\minuS\set{q^*,q_1,q_2}$, and as a result, the dimension  drops by $2$ at $q^*$ as the local condition at $q^*$ precedes the ones at $q_1$ or $q_2$ (see (\ref{eq:Figure2})).
In order to have the desired values of $\legendretwo{x}{\al_i}$ and $\legendretwo{y}{\al_j}$ as in (\ref{eq:x-al1}) and (\ref{eq:x-al2}), we require nontrivial values of some $\legendretwo{\al_i}{\primeP}$ for some $\primeP$ dividing $q_1$ or $q_2$. Thus, if the local condition at $q_1$ or $q_2$ preceded the one at $q^*$, some $\al_i\al_j$ might not survive the local condition at $q_1$ or $q_2$, and the dimension may not drop by $2$ at $q^*$.

Let us add a few words on using unramified primes with $n_p=2$.  The advantage of using this prime is that $\dim \HonE(\ratn,E)_{S_E} + 1 = \dim \HonE(\ratn,E_D)_{S_D}$ where $D=p$, and that  our technique yields a nontrivial map
$\Sel2(E) \to \HonE(\Qp,E[2])\unr$ especially if $\SEL(E)=1$.  This would result in $\SEL(E_D)=0$.
First of all, if $L/\ratn$ is Galois, then we do not have such a prime available.  Suppose that $L/\ratn$ is not Galois.
Our technique is basically to throw all the local conditions into a one ray class field $R$ of $L$, and use the Cebotarev density theorem and the law of quadratic reciprocity to choose $p$ and figure out what happens at $p$.  The following lemma summarizes the difficulty of choosing such $p$ with $n_p=2$:  Recall that $M$ is the Galois closure of $L$, and suppose that $\inner{\tau}=\Gal(M/L)$ and $\inner{\sig,\tau}=\Gal(M/\ratn)$.
\begin{lemma}
Let $R$ be a finite abelian extension of $L$, and $\al$ an automorphism in $\Gal(R/L)$. 
Let $\primeQ$ represent a prime ideal in $R$, and $\primeP$, a prime ideal in $L$.
Then, 
	there is $\Frob(\primeQ/\primeP)$ in $\Gal(R/L)$ such that $\Frob(\primeQ/\primeP)=\al$ and $f(\primeP/p)=2$
	if and only if there are a Galois extension $F/\ratn$ containing $RM$ and an automorphism $\mu$ in $\Gal(F/\ratn)$ such that 
	$\res_M(\mu)$ is $\tau\sig$ or $\tau\sig^2$ and $\res_R(\mu^2)=\al$.
	\end{lemma}
It seems to us that $\res_R(\mu^2)=\al$ is a very difficult condition to meet!  

Let us prove the corollary \ref{maincorollary}.
By Theorem \ref{maintheorem}, there is a quadratic twist of $E_{D_0}$ for which $\SEL(E_{D_0})$ is $0$ or $1$ and $D_0$ is odd.  Suppose that $\SEL(E_{D_0})=0$.
Using \cite{monsky:1996}, the parity conjecture of elliptic curves over $\ratn$ can be proved under the assumption on the finiteness of the Tate-Shafarevich group.
Since we construct a $D_0$  as a product of large odd primes, the quadratic twist $E_{D_0}$ has multiplicative  reduction at $v$.
Let $E':=E_{D_0}$, and let $\mathcal{H}_{Q}(n)$ be the Dirichlet character associated with the quadratic extension $\ratn(\sqrt{Q})$.  Recall that the root number $w(E'_D)$ is equal to $w(E') \cdot \mathcal{H}_D(-N_{E'})$ if the conductors $N_E$ and $N_{\mathcal{H}_D}$ are coprime to each other, where 
$N_{\mathcal{H}_D} = D$ if $D \equiv 1 \mod 4$. Recall that if $D\equiv 1 \mod 4$, then $\mathcal{H}_D(-N_{E'})=\mathcal{H}_D(N_{E'})$.  

Note that if $n_v=1$, then $E'$ must have good or additive reduction at $v$.
So, we consider only the two cases: $n_v=2$ or $3$. Suppose that $n_v=2$ or $3$ where $v$ is unramified in $L$.
Use Lemma \ref{key-lemma} to choose  a prime number $r$ with $n_r=1$, $\legendretwo{r}{v}=-1$, and $\legendretwo{r}{q}=1$ for all other primes $q$ of bad reduction of $E'$.
Since $\SEL(E')=0$, after we apply the local conditions of all $q\ne v$, the dimension of the following subgroup is $(n_v-1)$:
\begin{equation}\label{eq:last-subgroup}
\curlyC(E',v):=\set{ \al \in \HonE(\ratn,E'[2]) : \al \in \Img\delta_w,\ \text{for all } w \ne v }
\end{equation}
where $\delta_w$ are the local coboundary maps for $E'$.
Let $D=r$.  Then, $\curlyC(E'_D,v) = \curlyC(E',v)$ and, hence, $\SEL(E'_D)\le (n_v-1)$.
Note that since $E'$ has multiplicative reduction at $v$,  $\mathcal{H}_D(-N_{E'}) = -1$.  It follows that $w(E'_D)=-1$, and by \cite{monsky:1996}, $\rank(E'_D(\ratn))$ is odd which means that it is $1$. On the other hand, the finiteness of the Tate-Shafarevich group implies the non-degeneracy of the Casslels-Tate paring from which it follows that $\dim\sha(E'_D)[2]\ne 1$.
Thus, $\SEL(E'_D)=1$.

Suppose that $n_v=2$ and $v$ is ramified in $L$.
Consider the field extension $F=L(\sqrt{-1},\sqrt{q} : q \mid 6S' )$ where $S'$ is the product of all places of bad reduction of $E'$.  Let $\Frob(\primeP/p)$ be a Frobenius automorphism of $F/\ratn$ which trivially restricts to $L(\sqrt{-1})$ and $L(\sqrt{q})$ for all $q\ne v$, but nontrivially to $L(\sqrt{v})/L$. So, $\Frob(\primeP/p)$ nontrivially restricts to $M/L$ and, hence, $n_p=2$. Moreover, via the quadratic reciprocity, this means that $\legendretwo{p}{v}=-1$ and $\legendretwo{p}{q}=1$ for all $q \ne v$.  Let $D=p$. Then, $D \equiv 1 \mod 4$ and, hence, $\mathcal{H}_D(-N_{E'})=\mathcal{H}_D(N_{E'})=-1$.
Since  $\dim\curlyC(E',v) = 1$, $\dim\curlyC(E'_D,p) \le 1 + (n_p-1) = 2$ and, hence, $\SEL(E'_D) \le 2$. As  in the previous cases,$\SEL(E'_D)=1$ follows from  $w(E'_D)=-1$.  The result of the corollary follows from \cite[Theorem 1.2]{chang-thesis:2004}.

\renewcommand{\baselinestretch}{0.9}
\small
 \bibliography{smallselmer9_arxiv}
 \bibliographystyle{amsplain}

\end{document}